\pdfoutput=1 

\documentclass[11pt]{amsart}

\usepackage{amssymb,amsthm,enumitem,colonequals,tikz-cd,microtype}
\usepackage[normalem]{ulem} 

\usepackage[osf]{Baskervaldx}
\usepackage[baskervaldx]{newtxmath}
\usepackage[cal=boondoxo]{mathalfa} 

\usepackage[top=3.75cm, bottom=3cm, left=3.5cm, right=3.5cm]{geometry}

\usepackage{xcolor}
\colorlet{darkblue}{blue!55!black}
\colorlet{darkcyan}{cyan!50!black}
\colorlet{darkgreen}{green!60!black}

\PassOptionsToPackage{hyphens}{url}
\usepackage{hyperref}
\hypersetup{
    colorlinks=true,
    linkcolor=darkblue,
    urlcolor=darkcyan,
    citecolor=darkgreen,
}

\def\eqref#1{\textcolor{darkblue}{(\ref{#1})}}

\usepackage[nameinlink]{cleveref} 
\Crefformat{section}{#2\S#1#3}
\Crefmultiformat{section}{#2\S\S#1#3}{ and~#2#1#3}{, #2#1#3}{, and~#2#1#3}

\usepackage[pagewise]{lineno}
\let\oldequation\equation
\let\oldendequation\endequation
\renewenvironment{equation}{\linenomathNonumbers\oldequation}{\oldendequation\endlinenomath}
\expandafter\let\expandafter\oldequationstar\csname equation*\endcsname
\expandafter\let\expandafter\oldendequationstar\csname endequation*\endcsname
\renewenvironment{equation*}{\linenomathNonumbers\oldequationstar}{\oldendequationstar\endlinenomath}
\let\oldalign\align
\let\oldendalign\endalign
\renewenvironment{align}{\linenomathNonumbers\oldalign}{\oldendalign\endlinenomath}
\expandafter\let\expandafter\oldalignstar\csname align*\endcsname
\expandafter\let\expandafter\oldendalignstar\csname endalign*\endcsname
\renewenvironment{align*}{\linenomathNonumbers\oldalignstar}{\oldendalignstar\endlinenomath}

\theoremstyle{plain}
\newtheorem{theorem}{Theorem}[section]
\newtheorem{lemma}[theorem]{Lemma}
\newtheorem{corollary}[theorem]{Corollary}
\newtheorem{proposition}[theorem]{Proposition}

\theoremstyle{definition}
\newtheorem{definition}[theorem]{Definition}

\newtheorem{remark}[theorem]{Remark}

\newtheorem*{ack}{Acknowledgments}

\AddToHook{env/lemma/begin}{\crefalias{theorem}{lemma}}
\AddToHook{env/corollary/begin}{\crefalias{theorem}{corollary}}
\AddToHook{env/proposition/begin}{\crefalias{theorem}{proposition}}
\AddToHook{env/definition/begin}{\crefalias{theorem}{definition}}
\AddToHook{env/remark/begin}{\crefalias{theorem}{remark}}
\AddToHook{env/setup/begin}{\crefalias{theorem}{setup}}
\AddToHook{env/example/begin}{\crefalias{theorem}{example}}

\numberwithin{equation}{section}
\numberwithin{theorem}{section}

\title[Perfect generation for regular algebraic stacks]{Perfect generation for regular algebraic stacks}

\author[P.~Lank]{Pat Lank}
\address{P.~Lank,
Dipartimento di Matematica “F. Enriques”, Universit\`{a} degli Studi di Milano, Via Cesare
Saldini 50, 20133 Milano, Italy}
\email{plankmathematics@gmail.com}

\date{\today}



\begin{document}
 
\begin{abstract}
    We show a single perfect complex generates the derived category of complexes with quasi-coherent cohomology on a quasi-compact quasi-separated regular algebraic stack with quasi-finite diagonal. A key ingredient includes gluing generators along recollements.
\end{abstract} 

\maketitle


\section{Introduction}
\label{sec:intro}

\subsection{What is known}
\label{sec:intro_what_is_known}

Consider the derived category $D_{\operatorname{qc}}$ of complexes with quasi-coherent cohomology on an algebraic stack. Understanding generators of $D_{\operatorname{qc}}$ are important for studying its structure. It is classical that $D_{\operatorname{qc}}$ is singly compactly generated for quasi-compact quasi-separated schemes \cite[Theorem 3.1.1]{Bondal/VandenBergh:2003}. 

Early extensions for single object compact generation to various classes of algebraic stacks were obtained in \cite[Corollary 5.2]{Toen:2012}, \cite[proof of Proposition 5.5]{Krishna:2009}, and \cite[\S 3.3]{Ben-Zvi/Francis/Nadler:2010}. More recently, the work of Hall--Rydh has pushed these results further. These include quasi-compact quasi-separated algebraic stacks with quasi-finite separated diagonal \cite[Theorem A]{Hall/Rydh:2017a} and quasi-compact quasi-separated Deligne--Mumford $\mathbb{Q}$-stacks \cite[Theorem 7.4]{Hall/Rydh:2018}.

\subsection{What is new}
\label{sec:intro_what_is_new}

It was shown in \cite[Theorem 2.1]{Hall:2022} that $D_{\operatorname{qc}}$ is countably compactly generated for any concentrated \textit{regular} algebraic stack. It is natural to ask whether countable compact generation can be improved to generation by a single object. As a first step towards addressing this question, we use a weaker form of generation (see \Cref{def:generation_weak}).

\begin{theorem}
    \label{thm:generate}
    Let $\mathcal{X}$ be a quasi-compact quasi-separated regular algebraic stack with quasi-finite diagonal. Then $D_{\operatorname{qc}}(\mathcal{X})$ is generated by a single perfect complex.
\end{theorem}

In particular, \Cref{thm:generate} applies to cases for which compact generation is not necessarily known. For example, quasi-DM stacks which are smooth and finitely presented over a field or DVR (perhaps of mixed characteristic). Returning to the concentrated case:

\begin{corollary}
    \label{cor:generate}
    If $\mathcal{X}$ is concentrated, then $D_{\operatorname{qc}}(\mathcal{X})$ is singly compactly generated.
\end{corollary}

As a consequence, $\mathcal{X}$ and all of its closed substacks (which need not be regular) satisfy the $1$-Thomason condition of \cite{Hall/Rydh:2017a}. This provides new examples in arbitrary characteristic, e.g.\ if $\mathcal{X}$ has affine and nice stabilizers (see \cite[Theorem 2.1]{Hall/Rydh:2015}). Our arguments combines Nisnevich d\'{e}vissage for algebraic stacks with recollement techniques for gluing generators. To our knowledge, this approach has not previously been used to establish single object generation in this generality.

The proof of \Cref{thm:generate} proceeds in two steps. First, we use the presentations of \cite{Hall/Rydh:2018} for algebraic stacks with finite stabilizers. This allows us to induct on the length of a monomorphic splitting sequence of a Nisnevich covering. Second, we use recollement techniques for triangulated categories to glue generators (see \Cref{prop:stacky_recollement}). This reduces the problem to detecting single object generation on suitable locally closed substacks. The regularity hypothesis enters through a finite duality argument, which allows us to show that the relative categories $D_{\operatorname{qc},Z}$ admit single generators.

\begin{remark}
    It is not obvious whether \Cref{thm:generate} can be upgraded to single compact generation without imposing concentratedness. The reason is that compact generation need not always glue along recollements. 
\end{remark}

\begin{ack}
    Lank was supported by ERC Advanced Grant 101095900-TriCatApp. The author greatly appreciates comments and/or discussions with Jarod Alper, Timothy De Deyn, Jack Hall, Michal Hrbek, Kabeer Manali Rahul, Fei Peng, and David Rydh. The author thanks Hall for communicating the improvements of \cite{Hall:2022}. The author appreciates comments given by an anonymous referee.
\end{ack}

\section{Algebraic stacks}
\label{sec:algebraic_stacks}

Our conventions for algebraic stacks are those of \cite{StacksProject}. For the derived pullback and pushforward adjunction, we adopt the conventions of \cite[\S1]{Hall/Rydh:2017a} and \cite{Olsson:2007a,Laszlo/Olsson:2008a,Laszlo/Olsson:2008b}. Unless otherwise specified, symbols such as $X$, $Y$, etc.\ denote schemes or algebraic spaces, while $\mathcal{X}$, $\mathcal{Y}$, etc.\ denote algebraic stacks. In this section, let $\mathcal{X}$ be a quasi-compact quasi-separated algebraic stack. 

\subsubsection{Categories}
\label{sec:prelim_stacks_categories}

Let $\operatorname{Mod}(\mathcal{X})$ denote the Grothendieck abelian category of sheaves of $\mathcal{O}_\mathcal{X}$-modules on the lisse-\'{e}tale site of $\mathcal{X}$. Define $\operatorname{Qcoh}(\mathcal{X})$ to be the strictly full subcategory (i.e.\ full and closed under isomorphisms) of $\operatorname{Mod}(\mathcal{X})$ consisting of quasi-coherent sheaves. Set $D(\mathcal{X}) := D(\operatorname{Mod}(\mathcal{X}))$ for the derived category of $\operatorname{Mod}(\mathcal{X})$. Denote by $D_{\operatorname{qc}}(\mathcal{X})$ the full subcategory of $D(\mathcal{X})$ consisting of complexes with quasi-coherent cohomology. Finally, let $\operatorname{Perf}(\mathcal{X})$ denote the full subcategory of perfect complexes in $D_{\operatorname{qc}}(\mathcal{X})$. 

\subsubsection{Support}
\label{sec:prelim_stacks_support}

Let $M \in \operatorname{Qcoh}(\mathcal{X})$. Set $\operatorname{supp}(M) := p\big(\operatorname{supp}(p^\ast M)\big) \subseteq |\mathcal{X}|$ where $p \colon U \to \mathcal{X}$ is any smooth surjective morphism from a scheme. One can check that this definition is independent of the choice of $p$. Now, given any $E \in D_{\operatorname{qc}}(\mathcal{X})$, let
\begin{displaymath}
    \operatorname{supp}(E) := \bigcup_{j \in \mathbb{Z}} \operatorname{supp}\big(\mathcal{H}^j(E)\big) \subseteq |\mathcal{X}|.
\end{displaymath}
This subset of $|\mathcal{X}|$ is called the \textbf{support} of $E$. Given a open immersion $j\colon \mathcal{U} \to \mathcal{X}$, $D_{\operatorname{qc},Z}(\mathcal{X}):=\ker (\mathbf{L}j^\ast)$ where $Z=|\mathcal{X}|\setminus |\mathcal{U}|$. In particular,
\begin{displaymath}
    D_{\operatorname{qc},Z}(\mathcal{X}) = \{ E \in D_{\operatorname{qc}}(\mathcal{X}) : \operatorname{supp}(E)\subseteq Z \}.
\end{displaymath}

\begin{lemma}
    \label{lem:support_for_pullback}
    Let $f\colon \mathcal{Y}\to \mathcal{X}$ be a morphism of quasi-compact quasi-separated algebraic stacks. Consider an open immersion $j\colon \mathcal{U}\to \mathcal{X}$. Set $Z:= |\mathcal{X}|\setminus |\mathcal{U}|$. Then 
    \begin{displaymath}
        \mathbf{L}f^\ast (D_{\operatorname{qc},Z}(\mathcal{X}))\subseteq D_{\operatorname{qc},f^{-1}(Z)}(\mathcal{Y}).
    \end{displaymath}
\end{lemma}

\begin{proof}
    Let $E\in D_{\operatorname{qc},Z}(\mathcal{X})$. Denote by $f^\prime\colon f^{-1}(\mathcal{U}) \to \mathcal{U}$ the base change of $j$ along $f$. Since $\mathbf{L}j^\ast E\cong 0$, it follows that $\mathbf{L}(j\circ f^\prime)^\ast E\cong0$. Set $j^\prime\colon f^{-1}(\mathcal{U}) \to \mathcal{Y}$ the open immersion obtained by base change. Then $\mathbf{L}(f\circ j\prime)^\ast E\cong0$, and so, $\mathbf{L}f^\ast E \in D_{\operatorname{qc},f^{-1}(Z)}(\mathcal{Y})$.
\end{proof}

\begin{lemma}
    \label{lem:pushforward_of_support}
    Let $f\colon \mathcal{Y}\to \mathcal{X}$ be a finite morphism of Noetherian algebraic stacks. For any closed subset $Z\subseteq |\mathcal{Y}|$, one has $\mathbf{R}f_\ast D_{\operatorname{qc},Z}(\mathcal{Y})\subseteq D_{\operatorname{qc},f(Z)}(\mathcal{X})$.
\end{lemma}

\begin{proof}
    Since $f$ is finite, the problem is smooth local. So, we can reduce to Noetherian schemes. However, this is known, see e.g.\ \cite[Remark 23.46(2)]{Gortz/Wedhorn:2023}. 
\end{proof}

\subsubsection{Concentratedness}
\label{sec:prelim_stacks_concentrated}

A quasi-compact quasi-separated morphism of algebraic stacks is called \textbf{concentrated} if for every base change along a quasi-compact quasi-separated morphism, the derived pushforward has finite cohomological dimension. For instance, by \cite[Lemma 2.5(3)]{Hall/Rydh:2017a}, morphisms which are representable by algebraic spaces are concentrated. An algebraic stack is \textbf{concentrated} if its structure morphism to $\operatorname{Spec}(\mathbb{Z})$ is concentrated. See \cite[\S2, Lemma 2.5(5), \& Remark 4.6]{Hall/Rydh:2017a} for details.

\subsubsection{Perfect complexes}
\label{sec:prelim_stacks_perfects}

Perfect complexes may be defined on any ringed site \cite[\href{https://stacks.math.columbia.edu/tag/08G4}{Tag 08G4}]{StacksProject}, in particular on the lisse-\'{e}tale site of $\mathcal{X}$. A complex is \textbf{strictly perfect} if it is a bounded complex whose terms are direct summands of finite free modules; it is \textbf{perfect} if it is locally strictly perfect. Any compact object of $D_{\operatorname{qc}}(\mathcal{X})$ is a perfect complex and the support of a perfect complex has quasi-compact complement (see \cite[Lemmas 4.4 \& 4.8]{Hall/Rydh:2017a}).

\subsubsection{Thomason condition}
\label{sec:prelim_stacks_Thomason}

We say that $\mathcal{X}$ satisfies the \textbf{$\beta$-Thomason condition}, for some cardinal $\beta$, if $D_{\operatorname{qc}}(\mathcal{X})$ is compactly generated by a set of cardinality at most $\beta$, and if for every closed subset $Z \subseteq |\mathcal{X}|$ with quasi-compact complement there exists a perfect complex $P \in \operatorname{Perf}(\mathcal{X})$ such that $\operatorname{supp}(P) = Z$.

\section{Recollements}
\label{sec:recollement}

\subsection{Reminder}
\label{sec:recollement_reminder}

We briefly recall the notion of a recollement. See \cite[\S 1.4]{Beilinson/Berstein/Deligne/Gabber:2018} for details. A \textbf{recollement} is a commutative diagram of triangulated categories and exact functors of the form 
\begin{equation}
    \label{eq:recollement}
    \begin{tikzcd}
        {\mathcal{T}} && {\mathcal{K}} && {\mathcal{D}}
        \arrow["I"{description}, from=1-1, to=1-3]
        \arrow["{I_\lambda}"', bend right =25pt, from=1-3, to=1-1]
        \arrow["{I_\rho}", bend right =-25pt, from=1-3, to=1-1]
        \arrow["Q"{description}, from=1-3, to=1-5]
        \arrow["{Q_\lambda}"', bend right =25pt, from=1-5, to=1-3]
        \arrow["{Q_\rho}", bend right =-25pt, from=1-5, to=1-3]
    \end{tikzcd}
\end{equation}
satisfying:
\begin{itemize}
    \item $I_\lambda \dashv I \dashv I_\rho$ and $Q_\lambda \dashv Q \dashv Q_\rho$ (i.e.\ adjoint triples)
    \item $I, Q_\lambda, Q_\rho$ are fully faithful
    \item $\ker (Q)$ coincides with the strictly full subcategory on objects of the form $I(T)$ where $T\in \mathcal{T}$.
\end{itemize}
In such a case, there are distinguished triangles
\begin{displaymath}
    \begin{aligned}
        (Q_\lambda \circ Q )(E) & \to E \to (I \circ I_\lambda )(E) \to (Q_\lambda \circ Q) (E)[1],
        \\& (I \circ I_\rho) (E) \to E \to (Q_\rho \circ Q) (E) \to (I \circ I_\rho )(E)[1]
    \end{aligned}
\end{displaymath}
which are functorial in $\mathcal{K}$. In particular, the natural transformations between these functors are given by the (co)units of the relevant adjoint pairs. Since $Q_\lambda$, $Q$, $I$, and $I_\lambda$ are left adjoints, they preserve coproducts. 

\subsection{A special case}
\label{sec:recollement_stacks}

We record some results in setting of algebraic stacks. The following is well-known for schemes (see e.g.\ \cite[Theorem 1]{Jorgensen:2009}).

\begin{proposition}
    \label{prop:stacky_recollement}
    Let $\mathcal{X}$ be a quasi-compact quasi-separated algebraic stack. Suppose $j\colon \mathcal{U}\to \mathcal{X}$ is a quasi-compact open immersion. Set $Z:=|\mathcal{X}|\setminus |\mathcal{U}|$. There exists a recollement
    \begin{displaymath}
        \begin{tikzcd}
            {D_{\operatorname{qc}}(\mathcal{U})} && {D_{\operatorname{qc}}(\mathcal{X})} && {D_{\operatorname{qc},Z}(\mathcal{X})}
            \arrow["{\mathbf{R}j_\ast}"{description}, from=1-1, to=1-3]
            \arrow["{\mathbf{L}j^\ast}"', shift right=3, bend right = 12pt, from=1-3, to=1-1]
            \arrow["{j^\times}", shift left=3, bend right = -12pt, from=1-3, to=1-1]
            \arrow["{i^!}"{description}, from=1-3, to=1-5]
            \arrow["{i_\ast}"', shift right=3, bend right = 12pt, from=1-5, to=1-3]
            \arrow["{i^\ast}", shift left=3, bend right = -12pt, from=1-5, to=1-3]
        \end{tikzcd}
    \end{displaymath}
    where $i_\ast$ is the natural inclusion and $j^\times$ the right adjoint of $\mathbf{R}j_\ast$. Furthermore, for any closed $W\subseteq |\mathcal{X}|$ such that $D_{\operatorname{qc},|\mathcal{U}|\cap W}(\mathcal{U})$ is compactly generated, there exists a recollement
    \begin{displaymath}
        \begin{tikzcd}
            {D_{\operatorname{qc},|\mathcal{U}|\cap W}(\mathcal{U})} && {D_{\operatorname{qc},W}(\mathcal{X})} && {D_{\operatorname{qc},W\cap Z}(\mathcal{X})}
            \arrow["{\mathbf{R}j_\ast}"{description}, from=1-1, to=1-3]
            \arrow["{\mathbf{L}j^\ast}"', shift right=3, bend right = 12pt, from=1-3, to=1-1]
            \arrow["{j^!}", shift left=3, bend right = -12pt, from=1-3, to=1-1]
            \arrow["{i^!}"{description}, from=1-3, to=1-5]
            \arrow["{i_\ast}"', shift right=3, bend right = 12pt, from=1-5, to=1-3]
            \arrow["{i^\ast}", shift left=3, bend right = -12pt, from=1-5, to=1-3]
        \end{tikzcd}
    \end{displaymath}
    where $i_\ast$ is the natural inclusion.
\end{proposition}

\begin{proof}
    We prove the first claim. First, we spell out the existence of the needed functors. Recall that the right adjoint $j^\times$ of $\mathbf{R}j_\ast$ exists because $j$ is concentrated (see \cite[Theorem 4.14(1)]{Hall/Rydh:2017a}). By \cite[Example 1.2]{Hall/Rydh:2017b}, there is a Verdier localization
    \begin{displaymath}
        D_{\operatorname{qc},Z}(\mathcal{X}) \xrightarrow{i_\ast} D_{\operatorname{qc}}(\mathcal{X}) \xrightarrow{\mathbf{L}j^\ast} D_{\operatorname{qc}}(\mathcal{U}).
    \end{displaymath}
    Note that $\mathbf{R}j_\ast$ is right adjoint to $\mathbf{L}j^\ast$ on $D_{\operatorname{qc}}$ (see \cite[\S 1.3]{Hall/Rydh:2017a}). So, by \cite[Lemma 2.2(ii)]{Gao/Psaroudakis:2018}, $i_\ast$ must admit a right adjoint, which we denote by $i^!$ (loc.\ cit.\ pulls from \cite[Theorem 1.1]{Cline/Parshall/Scott:1988a} and \cite[Theorem 2.1]{Cline/Parshall/Scott:1988b}). It follows that 
    \begin{displaymath}
        D_{\operatorname{qc},Z}(\mathcal{X}) \xleftarrow{i^!} D_{\operatorname{qc}}(\mathcal{X}) \xleftarrow{\mathbf{R}j_\ast} D_{\operatorname{qc}}(\mathcal{U})
    \end{displaymath}
    is a Verdier localization sequence. 
    However, being that $j^\times$ is the right adjoint of $\mathbf{R}j_\ast$ on $D_{\operatorname{qc}}$, we can apply \cite[Lemma 2.2(ii)]{Gao/Psaroudakis:2018} once more to see that $i^!$ admits a right adjoint as well, which we denote by $i^\ast$. Tying things together, we have the required data for a recollement; i.e.\ a Verdier localization sequence which is a localization and colocalization sequence in the sense of \cite[\S 4]{Krause:2010}.

    We show the second claim This mimics the proof above but a few details need to be spelled out. In fact, the difference between the two claims is whether or not $D_{\operatorname{qc},W}$ is compactly generated. By \Cref{lem:support_for_pullback}, the restriction of $\mathbf{L}j^\ast$ to $D_{\operatorname{qc},W}(\mathcal{X})$ factors through $D_{\operatorname{qc},|\mathcal{U}|\cap W}(\mathcal{U})$. Moreover, the restriction of $\mathbf{R}j_\ast$ to $D_{\operatorname{qc},|\mathcal{U}|\cap W}(\mathcal{U})$ factors through $D_{\operatorname{qc},W}(\mathcal{X})$, which can be verified using \cite[Theorem 2.6(2)]{Hall/Rydh:2017a} and checking smooth locally.
    Hence, the adjoint pair $\mathbf{L}j^\ast$ and $\mathbf{R}j_\ast$ restricts to $D_{\operatorname{qc},W}(\mathcal{X})$ and $D_{\operatorname{qc},|\mathcal{U}|\cap W}(\mathcal{U})$. Furthermore, there is a Verdier localization
    \begin{displaymath}
        D_{\operatorname{qc},W\cap Z}(\mathcal{X}) \xrightarrow{i_\ast} D_{\operatorname{qc},W}(\mathcal{X}) \xrightarrow{\mathbf{L}j^\ast} D_{\operatorname{qc},|\mathcal{U}|\cap W}(\mathcal{U}).
    \end{displaymath}
    Indeed, the counit of $\mathbf{L}j^\ast$ and $\mathbf{R}j_\ast$ is a natural isomorphism on $D_{\operatorname{qc}}$, and so the restriction $\mathbf{L}j^\ast\colon D_{\operatorname{qc},W}(\mathcal{X}) \to D_{\operatorname{qc},|\mathcal{U}|\cap W}(\mathcal{U})$ is essentially surjective. To see that fully faithfulness holds, use e.g.\ \cite[\S I.\ Proposition 1.3]{Gabriel/Zisman:1967} and \cite[Lemma 4.3.1]{Krause:2010}). Note that restriction of $\mathbf{R}j_\ast$ is right adjoint to the restriction of $\mathbf{L}j^\ast$. So, by \cite[Lemma 2.2(ii)]{Gao/Psaroudakis:2018}, $i_\ast$ must admit a right adjoint as well, which we denote by $i^!$ (loc.\ cit.\ pulls from \cite[Theorem 1.1]{Cline/Parshall/Scott:1988a} and \cite[Theorem 2.1]{Cline/Parshall/Scott:1988b}). It follows that 
    \begin{displaymath}
        D_{\operatorname{qc},Z,\cap W}(\mathcal{X}) \xleftarrow{i^!} D_{\operatorname{qc},W}(\mathcal{X}) \xleftarrow{\mathbf{R}j_\ast} D_{\operatorname{qc},|\mathcal{U}|\cap W}(\mathcal{U})
    \end{displaymath}
    is a Verdier localization sequence. 
    However, $D_{\operatorname{qc},|\mathcal{U}|\cap W}(\mathcal{U})$ is compactly generated and the restriction of $\mathbf{R}j_\ast$ preserves small coproducts. Hence, \cite[Corollary 2.3]{Balmer/Dell'Ambrogio/Sanders:2016} tells us the restriction of $\mathbf{R}j_\ast$ admits a right adjoint. Now, we can apply \cite[Lemma 2.2(ii)]{Gao/Psaroudakis:2018} once more to see that $i^!$ admits a right adjoint as well, which we denote by $i^\ast$. Tying things together, we have the required data for a recollement; i.e.\ a Verdier localization sequence which is a localization and colocalization sequence in the sense of \cite[\S 4]{Krause:2010}. 
\end{proof}

\begin{remark}
    \label{rmk:stacky_recollement_compacts_preserved_by_ishriek}
    Consider the recollement in \Cref{prop:stacky_recollement}. In this case, $i^!$ preserves small coproducts, and so, $i_\ast$ preserves compact objects (see e.g.\ the proof of $\implies$ in \cite[Theorem 5.1]{Neeman:1996}; which this fact does not require compact generation). Then, from \cite[Lemma 4.4(1)]{Hall/Rydh:2017a}, it follows that any compact object of $D_{\operatorname{qc},Z}(\mathcal{X})$ must belong to $\operatorname{Perf}(\mathcal{X})$.
\end{remark}

\begin{remark}
    In \Cref{prop:stacky_recollement}, the restriction of $\mathbf{R}j_\ast$ to $D_{\operatorname{qc},|\mathcal{U}|\cap W}(\mathcal{U})\to D_{\operatorname{qc},W}(\mathcal{X})$ admits a right adjoint. However, it is not clear to us whether it is the restriction of the right adjoint $j^\times$ of $\mathbf{R}j_\ast$ on $D_{\operatorname{qc}}$.
\end{remark}

\subsection{Generating}
\label{sec:recollement_generating}

We record a few useful lemmas allowing one to glue generators along a recollement. 

\begin{definition}
    \label{def:generation_weak}
    Recall that a collection $\mathcal{B}$ in a category $\mathcal{C}$ is said to \textbf{generate} $\mathcal{C}$ if for any nonzero $E\in \mathcal{C}$ there is an $n\in \mathbb{Z}$ and $B\in\mathcal{B}$ such that $\operatorname{Hom}(B[n],E)\not\cong 0$.
\end{definition}

\begin{lemma}
    \label{lem:pullback_generator}
    Let $F\colon \mathcal{T} \leftrightarrows \mathcal{S} \colon G$ be an adjoint pair of exact functors between triangulated categories admitting small coproducts. Suppose $\mathcal{T}$ is generated by a collection $\mathcal{B}$. Then $G$ is conservative (i.e.\ $G(A)\cong 0 \implies A\cong 0$) if, and only if, $F(\mathcal{B})$ generates $\mathcal{S}$. In such a case, if $\mathcal{T}$ is generated by a set of cardinality $\leq\beta$ for some cardinal $\beta$, then $\mathcal{B}$ satisfies the same condition.
\end{lemma}

\begin{proof}
    This is known to a few but we spell it out for convienence. Let $E\in \mathcal{B}$ satisfy $\operatorname{Hom}(F(B),E[n])\cong 0$ for all $B\in \mathcal{B}$ and $n\in \mathbb{Z}$. From adjunction, we know that $\operatorname{Hom}(B, G (E)[n])\cong 0$. As $\mathcal{B}$ generates $\mathcal{T}$, it follows that $G(E)\cong 0$. However, $G$ being conservative implies $E\cong 0$. So, $F(\mathcal{B})$ generates $\mathcal{S}$. 
    
    Conversely, assume that $F(\mathcal{B})$ generates $\mathcal{S}$. Let $E\in \mathcal{S}$ satisfy $G(E)\cong 0$. By adjunction, it follows that $0\cong\operatorname{Hom}(F(B),E[n])$ for all $B\in \mathcal{B}$ and $n\in \mathbb{Z}$. Yet, the assumption implies $E\cong 0$. Hence, $G$ must be conservative.

    That the last claim holds follows from the proof above.
\end{proof}

\begin{corollary}
    \label{cor:pullback_compact_generator}
    Let $F\colon \mathcal{T} \leftrightarrows \mathcal{S} \colon G$ be an adjoint pair of exact functors between triangulated categories admitting small coproducts. Suppose $\mathcal{T}$ is compactly generated by a collection $\mathcal{B}$ and $G$ commutes with small coproducts. Then $G$ is conservative if, and only if, $F(\mathcal{B})$ compactly generates $\mathcal{S}$. In such a case, if $\mathcal{T}$ is compactly generated by a set of cardinality $\leq\beta$ for some cardinal $\beta$, then $\mathcal{B}$ satisfies the same condition.
\end{corollary}

\begin{proof}
    That $G$ commutes with small coproducts ensures that $F(\mathcal{T}^c)\subseteq \mathcal{S}^c$ (see e.g.\ the proof of $\implies$ in \cite[Theorem 5.1]{Neeman:1996}; which this fact does not require $\mathcal{T}$ to be compactly generated). Consequently, the desired claim follows from \Cref{lem:pullback_generator}.
\end{proof}

\begin{proposition}
    \label{prop:glue_along_recollement}
    Consider a recollement as in \eqref{eq:recollement}. Let $\mathcal{G}$ be a subcategory of $\mathcal{K}$ such that $I_\lambda (\mathcal{G})$ generates $\mathcal{T}$. If $\mathcal{G}^\prime$ is a subcategory which generates $\mathcal{D}$, then the collection of $G\oplus Q_\lambda(G^\prime)$ where $G\in \mathcal{G}$ and $G^\prime \in \mathcal{G}^\prime$ generates $\mathcal{K}$.
\end{proposition}

\begin{proof}
    Let $E\in \mathcal{K}$ satisfy $\operatorname{Hom}((G\oplus Q_\lambda(G^\prime)) [n],E)=0$ for all $n\in \mathbb{Z}$, $G\in \mathcal{G}$, and $G^\prime \in \mathcal{G}^\prime$. From the recollement above, there is a distinguished triangle 
    \begin{displaymath}
        (Q_\lambda \circ Q )(E) \to E \to (I \circ I_\lambda )(E) \to (Q_\lambda \circ Q) (E)[1].
    \end{displaymath}
    Here, in such a case, we have that $\operatorname{Hom}(Q_\lambda (G^\prime) [n] , E) \cong 0$ for all $n\in \mathbb{Z}$. So, adjunction tells us $\operatorname{Hom}(G^\prime [n] , Q(E))\cong 0$ for all $n\in \mathbb{Z}$. However, $\mathcal{G}^\prime$ generates $\mathcal{D}$, so $Q(E)\cong 0$. This implies the morphism $E \to (I \circ I_\lambda )(E)$ from the distinguished triangle above is an isomorphism. Once more, from adjunction, we then have that $\operatorname{Hom}( G [n] , (I \circ I_\lambda )(E))=0$ for all $n\in \mathbb{Z}$. So, $\operatorname{Hom}(I_{\lambda} (G) [n] , I_\lambda (E))=0$ via adjunction for any such $n$. However, $I_{\lambda} (\mathcal{G})$ generates $\mathcal{T}$, so $I_\lambda (E)$. Consequently, $E\cong 0$, which completes the proof.
\end{proof}

\section{Proofs}
\label{sec:proofs}

This section contains our main results.

\begin{lemma}
    \label{lem:descent_for_generation}
    Let $f\colon \mathcal{Y}\to \mathcal{X}$ be a finite morphism of Noetherian algebraic stacks. Then the adjoint pair $\mathbf{R}f_\ast$ and $f^\times$ restricts to $D_{\operatorname{qc}}(\mathcal{Y})$ and $D_{\operatorname{qc}, f(|\mathcal{Y}|)}(\mathcal{X})$. Additionally, if $D_{\operatorname{qc}}(\mathcal{Y})$ is generated by some $\mathcal{B}\subseteq D^b_{\operatorname{coh}}(\mathcal{X})$ and $\mathbf{R}f_\ast \mathcal{O}_{\mathcal{Y}}\in \operatorname{Perf}(\mathcal{X})$, then $\mathbf{R}f_\ast \mathcal{B}$ generates $D_{\operatorname{qc}, f(|\mathcal{Y}|)}(\mathcal{X})$.
\end{lemma}

\begin{proof}
    By \cite[Theorem 4.14(1)]{Hall/Rydh:2017a}, $\mathbf{R}f_\ast$ admits a right adjoint $f^\times$ on $D_{\operatorname{qc}}$.
    Clearly, $f^\times (D_{\operatorname{qc}, f(|\mathcal{Y}|)}(\mathcal{X}))\subseteq D_{\operatorname{qc}}(\mathcal{Y})$. Moreover, \Cref{lem:pushforward_of_support} says $\mathbf{R}f_\ast (D_{\operatorname{qc}}(\mathcal{Y}))\subseteq D_{\operatorname{qc}, f(|\mathcal{Y}|)}(\mathcal{X})$.
    Thus, the adjunction restricts.

    Lastly, we show that $\mathbf{R}f_\ast \mathcal{B}$ generates $D_{\operatorname{qc}, f(|\mathcal{Y}|)}(\mathcal{X})$. By \Cref{lem:pullback_generator}, it suffices to show that $f^\times\colon D_{\operatorname{qc}, f(|\mathcal{Y}|)}(\mathcal{X}) \to D_{\operatorname{qc}}(\mathcal{Y})$ is conservative. So, let $E\in D_{\operatorname{qc},f(|\mathcal{Y}|)}(\mathcal{X})$ satisfy $f^\times E\cong 0$. By \cite[Theorem 4.14(2)]{Hall/Rydh:2017a}, it follows that 
    \begin{displaymath}
        0 \cong \mathbf{R}f_\ast f^\times E \cong \operatorname{\mathbf{R}\mathcal{H}\! \mathit{om}} (\mathbf{R}f_\ast \mathcal{O}_{\mathcal{Y}}, E).
    \end{displaymath}
    However, $\mathbf{R}f_\ast \mathcal{O}_{\mathcal{Y}}\in \operatorname{Perf}(\mathcal{X})\cap D_{\operatorname{qc},f(|\mathcal{Y}|)}(\mathcal{X})$, and so, \cite[Lemma 4.9]{Hall/Rydh:2017a} ensures that $E\cong 0$. Hence, we see that the restriction of $f^\times$ on $D_{\operatorname{qc},f(|\mathcal{Y}|)}(\mathcal{X})$ is conservative.
\end{proof}

\begin{lemma}
    [cf.\ {\cite[Observation 5.6]{Neeman:2023}}]
    \label{lem:injective_mono_is_thomason}
    Let $t\colon \mathcal{Y}\to \mathcal{X}$ be a quasi-affine morphism of quasi-compact quasi-separated algebraic stacks. Suppose $t$ is injective on the underlying topological spaces. If $\mathcal{X}$ satisfies the $\beta$-Thomason condition, then so does $\mathcal{Y}$. 
\end{lemma}

\begin{proof}
    As $t$ is quasi-affine, \cite[Lemma 8.2]{Hall/Rydh:2017a} says that $D_{\operatorname{qc}}(\mathcal{Y})$ is compactly generated by a collection of at most $\leq \beta$ objects. Next, let $Z$ be a closed subset of $|\mathcal{Y}|$. Set $Z^{\prime}$ to be the closure of $t(Z)$ in $|\mathcal{X}|$. Then $t^{-1}(Z^\prime)=Z$. Moreover, the hypothesis on $\mathcal{X}$ allows us to find a $P\in \operatorname{Perf}(\mathcal{X})$ such that $\operatorname{supp}(P)=Z^\prime$. By \cite[Lemma 4.8(2)]{Hall/Rydh:2017a}, $t^{-1}(\operatorname{supp}(P))=\operatorname{supp}(\mathbf{L}t^\ast P)$.
\end{proof}

\begin{proof}
    [Proof of \Cref{thm:generate}]
    The proof proceeds by an inductive argument. However, to set the stage, we start with a construction needed for such an argument. By \cite[Theorem 4.1]{Hall/Rydh:2018}, there exist morphisms of algebraic stacks $V\xrightarrow{p} \mathcal{Y} \xrightarrow{f} \mathcal{X}$ satisfying
    \begin{itemize}
        \item $V$ is an affine scheme
        \item $p$ is finite, flat, surjective, and of finite presentation
        \item $f$ is a Nisnevich covering (see \cite[Definition 3.1]{Hall/Rydh:2018}) of finite presentation with separated diagonal (in fact, $f$ is \'{e}tale and surjective). 
    \end{itemize}
    Moreover, from \cite[Proposition 3.1]{Hall/Rydh:2018}, there exists a monomorphic splitting sequence for $f$ in the sense of loc.\ cit.\ Specfically, we have a sequence of open immersions 
    \begin{displaymath}
        \emptyset=:\mathcal{X}_0 \xrightarrow{j_0} \mathcal{X}_1 \xrightarrow{j_1}\cdots \xrightarrow{j_{n-2}} \mathcal{X}_{n-1}\xrightarrow{j_{n-1}} \mathcal{X}_n =:\mathcal{X}
    \end{displaymath}
    such that the base change of $f$ to $|\mathcal{X}_c| \setminus |\mathcal{X}_{c-1}|$ (when given the reduced closed substack structure) admits a monomorphic section for each $c=1,\ldots, n$.
    
    Next, before going to the induction argument, we make an observation for components of the construction above. Fix $c=1,\ldots, n$. Consider the fibered square
    \begin{displaymath}
        \begin{tikzcd}
            {\mathcal{Y}_c} & {(|\mathcal{X}_c| \setminus |\mathcal{X}_{c-1}|)_{red}} \\
            {\mathcal{Y}} & {\mathcal{X}}
            \arrow["{f_c}", from=1-1, to=1-2]
            \arrow["{h_c}"', from=1-1, to=2-1]
            \arrow["{j^\prime_c \circ t_c}", from=1-2, to=2-2]
            \arrow["f"', from=2-1, to=2-2]
        \end{tikzcd}
    \end{displaymath}
    obtained by base change of $f$ along $j^\prime_c \circ t_c$ where $j^\prime_c$ is the associated open immersion and $t_c$ the associated closed immersion. Choose a monomorphic section $g_c \colon (|\mathcal{X}_c| \setminus |\mathcal{X}_{c-1}|)_{red} \to \mathcal{Y}_c$. This gives us a commutative diagram 
    \begin{displaymath}
        \begin{tikzcd}
            {(|\mathcal{X}_c| \setminus |\mathcal{X}_{c-1}|)_{red}} & {\mathcal{Y}_c} \\
            & {(|\mathcal{X}_c| \setminus |\mathcal{X}_{c-1}|)_{red}.}
            \arrow["{g_c}", from=1-1, to=1-2]
            \arrow["{\textrm{id.}}"', from=1-1, to=2-2]
            \arrow["{f_c}", from=1-2, to=2-2]
        \end{tikzcd}
    \end{displaymath}
    Base change tells us that $f_c$ is \'{e}tale and of finite presentation. Hence, $g_c$ must be \'{e}tale (see e.g.\ \cite[\href{https://stacks.math.columbia.edu/tag/0CIR}{Tag 0CIR}]{StacksProject}). Furthermore, since $g_c$ is monomorphic, it is separated, of finite type, and representable by algebraic spaces (see e.g.\ \cite[\href{https://stacks.math.columbia.edu/tag/06MY}{Tag 06MY} \& \href{https://stacks.math.columbia.edu/tag/050J}{Tag 050J}]{StacksProject}). Yet, each algebraic stack in our proof is Noetherian, so $g_c$ must be of finite presentation (see e.g.\ \cite[\href{https://stacks.math.columbia.edu/tag/0DQJ}{Tag 0DQJ}]{StacksProject}). It follows that $g_c$ must be quasi-affine (use e.g.\ \cite[Proposition 3.1]{Olsson/Starr:2003}). There is a fibered square,
    \begin{displaymath}
        \begin{tikzcd}
            {V_c} & {\mathcal{Y}_c} \\
            V & {\mathcal{Y}.}
            \arrow["{p_c}", from=1-1, to=1-2]
            \arrow["{h^\prime_c}"', from=1-1, to=2-1]
            \arrow["{h_c}", from=1-2, to=2-2]
            \arrow["p"', from=2-1, to=2-2]
        \end{tikzcd}
    \end{displaymath} 
    Again, from base change, we know that $h_c$, and hence $h^\prime_c$, is an immersion. Since any such morphism is quasi-affine, it follows that $V_c$ is a quasi-affine scheme. Hence, by \cite[Theorem C]{Hall/Rydh:2017a}, we know that $\mathcal{Y}_c$ must be $1$-Thomason because $p_c$ is a finite, flat, and surjective morphism of finite presentation from a quasi-affine scheme. Thus, from \Cref{lem:injective_mono_is_thomason}, we know that each $(|\mathcal{X}_c| \setminus |\mathcal{X}_{c-1}|)_{red}$ must be $1$-Thomason because $g_c$ is a quasi-affine monomorphism morphism (see e.g.\ \cite[\href{https://stacks.math.columbia.edu/tag/0500}{Tag 0500}]{StacksProject}). As each $\mathcal{X}_c$ is regular, \Cref{lem:descent_for_generation} tells us that $D_{\operatorname{qc},Z_c}(\mathcal{X_c})$ is singly compactly generated where $Z_c := t_c (|(|\mathcal{X}_c| \setminus |\mathcal{X}_{c-1}|)_{red}|)$. We can find $P_c \in \operatorname{Perf}(\mathcal{X}_c)$ which compactly generates $D_{\operatorname{qc},Z_c}(\mathcal{X_c})$.

    Now, we prove the desired claim by induction on $n$. Specifically, we are inducting the length of the monomorphic splitting sequence of $f$. There is nothing to check if $n=0$. Moreover, the observation above addresses the base case $n=1$. So, we may assume that $n\geq 2$. By \cite[Proposition B.1]{Hall/Lamarche/Lank/Peng:2025}, there is a Verdier localization sequence
    \begin{displaymath}
        D^b_{\operatorname{coh},Z_c}(\mathcal{X}_c) \xrightarrow{(i_c)_\ast} D^b_{\operatorname{coh}}(\mathcal{X}_c) \xrightarrow{\mathbf{L}j_{c-1}^\ast} D^b_{\operatorname{coh}}(\mathcal{X}_{c-1})
    \end{displaymath}
    where $(i_c)_\ast$ is the natural inclusion. However, \cite[Theorem 3.7]{DeDeyn/Lank/ManaliRahul/Peng:2025} tells us that $\operatorname{Perf}=D^b_{\operatorname{coh}}$ in each case. This allows us to find a $G_c\in \operatorname{Perf}(\mathcal{X}_c)$ such that $\mathbf{L}j_{c-1}^\ast G_c$ generates $D_{\operatorname{qc}}(\mathcal{X}_{c-1})$. Consider the recollement obtained in \Cref{prop:stacky_recollement},
    \begin{displaymath}
        \begin{tikzcd}
            {D_{\operatorname{qc}}(\mathcal{X}_{c-1})} && {D_{\operatorname{qc}}(\mathcal{X}_c)} && {D_{\operatorname{qc},Z_c}(\mathcal{X}_c)}
            \arrow["{\mathbf{R}(j_{c-1})_\ast}"{description}, from=1-1, to=1-3]
            \arrow["{\mathbf{L}j^\ast_{c-1}}"', shift right=3, bend right = 12pt, from=1-3, to=1-1]
            \arrow["{j^\times_c}", shift left=3, bend right = -12pt, from=1-3, to=1-1]
            \arrow["{i^!_c}"{description}, from=1-3, to=1-5]
            \arrow["{(i_c)_\ast}"', shift right=3, bend right = 12pt, from=1-5, to=1-3]
            \arrow["{i_c^\ast}", shift left=3, bend right = -12pt, from=1-5, to=1-3]
        \end{tikzcd}
    \end{displaymath}
    where $(i_c)_\ast$ is the natural inclusion and $j^\times_c$ the right adjoint of $\mathbf{R}j_\ast$. So, in particular, $(i_c)_\ast P_c = P_c$. Then, from \Cref{prop:glue_along_recollement}, we know that $G_c\oplus P_c$ generates $D_{\operatorname{qc}}(\mathcal{X}_c)$. This establishes the induction step. Moreover, since $f$ always has monomorphic splitting sequence of finite length, we complete proof. 
\end{proof}

\begin{proof}
    [Proof of \Cref{cor:generate}]
    This follows from \Cref{thm:generate} and the fact that $\mathcal{X}$ is concentrated.
\end{proof}

\bibliographystyle{alpha}
\bibliography{mainbib}

\end{document}